\documentclass[reqno,a4paper]{amsart}%
\usepackage{amssymb,url}
\usepackage[hypertex]{hyperref}
\newtheorem{thm}{Theorem}
\newtheorem{lem}{Lemma}
\theoremstyle{remark}
\newtheorem{rem}{Remark}
\DeclareMathOperator{\td}{d\mspace{-2mu}}

\begin{document}

\title[Some bounds for the complete elliptic integrals]
{Some bounds for the complete elliptic integrals of the first and second kinds}

\author[B.-N. Guo]{Bai-Ni Guo}
\address[B.-N. Guo]{School of Mathematics and Informatics,
Henan Polytechnic University, Jiaozuo City, Henan Province, 454010, China}
\email{\href{mailto: B.-N. Guo <bai.ni.guo@gmail.com>}{bai.ni.guo@gmail.com}, \href{mailto: B.-N. Guo <bai.ni.guo@hotmail.com>}{bai.ni.guo@hotmail.com}}
\urladdr{\url{http://guobaini.spaces.live.com}}

\author[F. Qi]{Feng Qi}
\address[F. Qi]{Research Institute of Mathematical Inequality Theory, Henan Polytechnic University, Jiaozuo City, Henan Province, 454010, China} \email{\href{mailto: F. Qi
<qifeng618@gmail.com>}{qifeng618@gmail.com}, \href{mailto: F. Qi
<qifeng618@hotmail.com>}{qifeng618@hotmail.com}, \href{mailto: F. Qi
<qifeng618@qq.com>}{qifeng618@qq.com}}
\urladdr{\url{http://qifeng618.spaces.live.com}}


\begin{abstract}
In the article, the complete elliptic integrals of the first and second kinds are bounded by using the power series expansions of some functions, the celebrated Wallis' inequality, and an integral inequality due to R. P. Agarwal, P. Cerone, S. S. Dragomir and F. Qi.
\end{abstract}

\thanks{This paper was typeset using \AmS-\LaTeX}

\maketitle

\section{Introduction}

The complete elliptic integrals of the first and second kinds may be defined \cite[pp.~590--592]{abram} as
\begin{equation}\label{elliptic-1st-dfn}
E(t)=\int_0^{\pi/2}\sqrt{1-t^2\sin^2\theta}\,\td\theta
\end{equation}
and
\begin{equation}\label{elliptic-2nd-dfn}
F(t)=\int_0^{\pi/2}\frac{\td\theta}{\sqrt{1-t^2\sin^2\theta}}
\end{equation}
for $0<t<1$. They can also be defined by
\begin{equation}\label{elliptic-1st-dfn-ab}
E(a,b)=\int_0^{\pi/2}\sqrt{a^2\cos^2\theta+b^2\sin^2\theta}\,\td \theta
\end{equation}
and
\begin{equation}\label{elliptic-2nd-dfn-ab}
F(a,b)=\int_0^{\pi/2}\frac{\td\theta}{\sqrt{a^2\cos^2\theta+b^2\sin^2\theta}\,}
\end{equation}
for positive numbers $a$ and $b$.
\par
It is not difficult to see that if $a>b>0$ then
\begin{equation}
E(a,b)=aE\Biggl(\sqrt{1-\frac{b^2}{a^2}}\,\Biggr)\quad\text{and}\quad F(a,b)=\frac1aF\Biggl(\sqrt{1-\frac{b^2}{a^2}}\,\Biggr).
\end{equation}
Conversely, if $0<t<1$, then
\begin{equation}
E(t)=E\bigl(1,\sqrt{1-t^2}\,\bigr)\quad \text{and}\quad F(t)=F\bigl(1,\sqrt{1-t^2}\,\bigr).
\end{equation}
\par
For more information on the history, background, properties and applications, please refer to \cite{Almkvist-Berndt} and related references therein.
\par
The aim of this paper is to establish several double inequalities for bounding the complete elliptic integrals $E(t)$, $F(t)$ and $F(a,b)$.
\par
Our main results are stated in the following theorems.

\begin{thm}\label{1st-bounds-one}
For $0<t<1$, we have
\begin{equation}\label{1st-upper-ineq}
\frac\pi2-\frac12\ln\frac{(1+t)^{1-t}}{(1-t)^{1+t}} <E(t)<\frac{\pi-1}2+\frac{1-t^2}{4t}\ln\frac{1+t}{1-t}.
\end{equation}
\end{thm}

\begin{thm}\label{2nd-elliptic-bounds-1}
For $b>a>0$,
\begin{equation}\label{Qi-1995-Xuebao-b>a}
\frac\pi2\cdot\frac{\ln\bigl(\sqrt{b/a}\,+\sqrt{b/a-1}\,\bigr)}{\sqrt{b(b-a)}\,} \le F(a,b) \le\frac\pi2\cdot\frac{\arctan\sqrt{b/a-1}\,}{\sqrt{a(b-a)}\,}.
\end{equation}
\end{thm}

\begin{thm}\label{iyengar-elliptic-ineq-thm}
For $0<t<1$, we have
\begin{multline}\label{iyengar-elliptic-ineq}
\biggl|\frac2\pi E(t)-\frac{1+\sqrt{1-t^2}\,}2\biggr|\\ \le\frac1\pi\Bigl(1-\sqrt{1-t^2}\,\Bigr)\left[1-\frac2\pi\cdot \frac{\sqrt{\bigl(1-t^2+\sqrt{1-t^2}\,\bigr)\bigl(1+\sqrt{1-t^2}\,\bigr)}\,} {\bigl(\sqrt{1-t^2}\,+1\bigr)\sqrt[4]{1-t^2}\,}\right].
\end{multline}
\end{thm}

\begin{thm}\label{iyengar-elliptic-2nd-ineq-thm}
For $0<t<1$, we have
\begin{multline}\label{iyengar-elliptic-2nd-ineq}
\biggl\vert\frac2\pi F(t) -\frac{\sqrt{1-t^2}\,+1}{2\sqrt{1-t^2}\,}\biggr\vert
\le\frac1\pi\cdot\frac{1-\sqrt{1-t^2}\,}{\sqrt{1-t^2}\,}\\
\times\left[1-\frac2\pi\cdot \frac{\bigl(1-\sqrt{1-t^2}\,\bigr)\bigl(2-t^2-\sqrt{t^4-t^2+1}\,\bigr)^{3/2}} {\sqrt{\bigl(1-t^2\bigr)\bigl(\sqrt{t^4-t^2+1}\,+t^2-1\bigr)\bigl(1-\sqrt{t^4-t^2+1}\,\bigr)}\,}\right].
\end{multline}
\end{thm}

\section{Lemmas}

In order to prove our main results, the following lemmas are necessary.

\begin{lem}
For $i\in\mathbb{N}$, we have
\begin{equation}\label{wallis-since-formula}
\int_0^{\pi/2}\sin^{2i}\theta\td \theta=\frac1{4^i}\binom{2i}{i}\frac\pi2.
\end{equation}
\end{lem}

\begin{proof}
This follows easily from Wallis sine formula \cite{WallisFormula.html}:
\begin{equation}\label{wsc}
\int_0^{\frac\pi2}\sin^nx\td x=
\begin{cases}
\dfrac{\pi}2\cdot\dfrac{(n-1)!!}{n!!}&\text{for $n$ even},\\[1em]
\dfrac{(n-1)!!}{n!!}&\text{for $n$ odd},
\end{cases}
\end{equation}
where $n!!$ denotes a double factorial.
\end{proof}

\begin{lem}[{\cite[p.~98, (2.12)]{3rded}}]
For $i\in\mathbb{N}$, we have
\begin{equation}\label{binom-ineq-accurate}
\frac{4^i}{\sqrt{\pi(i+1/2)}}<\binom{2i}i<\frac{4^i}{\sqrt{\pi(i+1/4)}}.
\end{equation}
\end{lem}

\begin{rem}
Bounding Wallis' formula \eqref{wsc} such as \eqref{binom-ineq-accurate} has a long history. For more information, please refer to related contents in the books \cite{3rded} and \cite[pp.~192--193, p.~287]{mit}. In \cite{Chen-CMA-1-1-06, wallis-chen-gen-math, wallis3, wallis-indon, wallis-tamk, wallis2, chenwallis, chenwallis-rgmia, wallis-gaz, construct, wallis-sun-qu} and \cite[Theorem~2]{G.-M.-Zhang}, the double inequality \eqref{binom-ineq-accurate} and its sharpness were recovered, proved and refined once and again because of either without being aware of and finding out the original version of the paper \cite{waston}, or making use of various approaches and subtle techniques, or repeating some existed routines. But, most of them were not devoted to improve the bounds in \eqref{binom-ineq-accurate}, except \cite{wallis-cao, Koumandos-PAMS-06, Zhao-De-Jun, zhao-wu-cn, zhao-wu} and \cite[Theorem~1]{G.-M.-Zhang}. Actually it was said early in \cite{Kazarinoff-56} that it is unquestionable that inequalities similar to \eqref{binom-ineq-accurate} can be improved indefinitely but at a sacrifice of simplicity.
\end{rem}

\begin{lem}
For $|s|<1$, we have
\begin{equation}\label{1+1/2-sum}
\sum_{i=0}^\infty\frac{s^{2i}}{i+1/2}=
\begin{cases}
\dfrac1s\ln\dfrac{1+s}{1-s},& s\ne0;\\
1,&s=0.
\end{cases}
\end{equation}
\end{lem}

\begin{proof}
This can be deduced readily from the power series expansions of the functions $\ln(1\pm s)$ at $s=0$.
\end{proof}

\begin{lem}[\cite{rpss, cer1, cer1-rgmia, gaz}]\label{gazthm2}
Let $f(x)$ be continuous on $[a, b]$ and differentiable in $(a,b)$. Suppose
that $f(x)$ is not identically a constant, and that $m\le f'(x)\le M$ in
$(a,b)$. Then
\begin{equation}\label{gaziy11}
\biggl\vert\frac1{b-a}\int_a^bf(x)\td x-\frac{f(a)+f(b)}2\biggr\vert
\le-\frac{[M-S_0(a,b)][m-S_0(a,b)]}{{2}(M-m)},
\end{equation}
where
\begin{equation}
  S_0(a,b)=\frac{f(b)-f(a)}{b-a}.
\end{equation}
\end{lem}

For more information on refinements and generalizations of \eqref{gaziy11}, please refer to \cite{iyengar-rocky-2005, iyengar-rocky-2005-rgmia} and related references therein.

\section{Proofs of main results}

\begin{proof}[Proof of Theorem~\ref{1st-bounds-one}]
It is general knowledge that
\begin{equation}\label{1-x-expansion}
\sqrt{1-x}\,=1-\sum_{i=0}^\infty\frac{(2i)!}{2^{2i+1}(i+1)(i!)^2}x^{i+1}
\end{equation}
for $0<x<1$. Hence, by replacing $x$ by $t^2\sin^2\theta$ in \eqref{1-x-expansion}, integrating on both sides with respect to $\theta\in\bigl[0,\frac\pi2\bigr]$ and making use of \eqref{wallis-since-formula}, we have
\begin{equation}\label{E(t)-expansion}
\begin{aligned}
E(t)&=\int_0^{\pi/2}\Biggl[1-\sum_{i=0}^\infty\frac{(2i)!}{2^{2i+1}(i+1)(i!)^2} t^{2(i+1)}\sin^{2(i+1)}\theta\Biggr]\td\theta\\
&=\frac\pi2-\sum_{i=0}^\infty\frac{(2i)!}{2^{2i+1}(i+1)(i!)^2} t^{2(i+1)}\int_0^{\pi/2}\sin^{2(i+1)}\theta\td\theta\\
&=\frac\pi2-\sum_{i=0}^\infty\frac{(2i)!}{2^{2i+1}(i+1)(i!)^2} \frac1{4^{i+1}}\binom{2i+2}{i+1}\frac\pi2t^{2(i+1)}\\
&=\frac\pi2\Biggl[1-\sum_{i=0}^\infty\frac1{4^{2i+2}(2i+1)} \binom{2i+2}{i+1}^2t^{2(i+1)}\Biggr]\\
&=\frac\pi2\Biggl[1-\sum_{i=1}^\infty\frac1{4^{2i}(2i-1)} \binom{2i}{i}^2t^{2i}\Biggr].
\end{aligned}
\end{equation}
Substituting \eqref{binom-ineq-accurate} into \eqref{E(t)-expansion} gives
\begin{equation}
\frac\pi2-2\sum_{i=1}^\infty\frac{t^{2i}}{(2i-1)(4i+1)} <E(t)<\frac\pi2-\sum_{i=1}^\infty\frac{t^{2i}}{(2i-1)(2i+1)}.
\end{equation}
\par
Straightforward computation and utilization of \eqref{1+1/2-sum} gives
\begin{multline*}
\sum_{i=1}^\infty\frac{t^{2i}}{(2i-1)(2i+1)}=\frac12\sum_{i=1}^\infty \biggl(\frac1{2i-1}-\frac1{2i+1}\biggr)t^{2i}\\
=\frac12+\frac{t^2-1}4\sum_{i=0}^\infty\frac{t^{2i}}{i+1/2}
=\frac12+\frac{t^2-1}{4t}\ln\frac{1+t}{1-t}.
\end{multline*}
This means the right-hand side inequality in \eqref{1st-upper-ineq}.
\par
By the similar argument as above, it follows that
\begin{multline}
\sum_{i=1}^\infty\frac{t^{2i}}{(2i-1)(4i+1)}<\frac14\sum_{i=1}^\infty\frac{t^{2i}}{i(2i-1)}
=\frac12\sum_{i=1}^\infty\biggl(\frac1{2i-1}-\frac1{2i}\biggr)t^{2i}\\
=\frac12\sum_{i=0}^\infty\frac{t^{2i+2}}{2i+1}-\frac14\sum_{i=1}^\infty\frac{t^{2i}}i
=\frac{t}{4}\ln\frac{1+t}{1-t}-\frac14\ln\bigl(1-t^2\bigr)
=\frac14\ln\frac{(1+t)^{1-t}}{(1-t)^{1+t}}.
\end{multline}
The left-hand side inequality in \eqref{1st-upper-ineq} follows. The proof of Theorem~\ref{1st-bounds-one} is complete.
\end{proof}

\begin{proof}[Proof of Theorem~\ref{2nd-elliptic-bounds-1}]
In \cite{ellip-xue-bao}, by discussing
\begin{equation*}
\sqrt{1+t^2\cos^2\theta}\,-\sqrt{1+t^2}\,+\frac4{\pi^2}
\Bigl(\sqrt{1+t^2}\,-1\Bigr)\theta^2+\alpha\biggl(\frac\pi2-\theta\biggr)\theta
\end{equation*}
or
\begin{equation*}
\sqrt{1+t^2\cos^2\theta}\,-\sqrt{1+t^2}\,+\frac2{\pi}
\Bigl(\sqrt{1+t^2}\,-1\Bigr)\theta+\beta\biggl(\frac\pi2-\theta\biggr)\theta
\end{equation*}
on $\bigl[0,\frac\pi2\bigr]$, where $\alpha$ and $\beta$ are undetermined constants, the inequality
\begin{multline}\label{81-ineq}
-\frac8{\pi^2}\Bigl(\sqrt{1+t^2}\,-1\Bigr)\theta\biggl(\frac\pi2-\theta\biggr)\le\\*
\sqrt{1+t^2\cos^2\theta}\,-\biggl[\sqrt{1+t^2}\,-\frac4{\pi^2} \Bigl(\sqrt{1+t^2}\,-1\Bigr)\theta^2\biggr]\le0
\end{multline}
for $\theta\in\bigl[0,\frac\pi2\bigr]$ was obtained, which is equivalent to
\begin{multline}
\frac1{\sqrt{1+t^2}\,-4\bigl(\sqrt{1+t^2}\,-1\bigr)\theta^2/\pi^2} \le\frac1{\sqrt{1+t^2\cos^2\theta}\,}\\ \le\frac1{\sqrt{1+t^2}\,-4\bigl(\sqrt{1+t^2}\,-1\bigr)\theta^2/\pi^2 -8\bigl(\sqrt{1+t^2}\,-1\bigr)\theta(\pi/2-\theta)/\pi^2}.
\end{multline}
Integrating on both sides of the above double inequality with respect to $\theta\in\bigl[0,\frac\pi2\bigr]$ yields
\begin{multline}\label{Qi-1995-Xuebao}
\frac\pi2\cdot\frac{\ln\Bigl(\sqrt[4]{1+t^2}\,+\sqrt{\sqrt{1+t^2}\,-1}\,\Bigr)} {\sqrt{1+t^2-\sqrt{1+t^2}\,}\,}
\le\int_0^{\pi/2}\frac{\td\theta}{\sqrt{1+t^2\cos^2\theta}\,} \\ \le\int_0^{\pi/2}\frac{\td\theta}{\bigl(\sqrt{t^2+1}\,-1\bigr)(2\theta/\pi-1)^2+1} =\frac\pi2\cdot\frac{\arctan{\sqrt{\sqrt{1+t^2}\,-1}\,}} {\sqrt{\sqrt{t^2+1}\,-1}\,}.
\end{multline}
Replacing $t^2$ by $\frac{b^2}{a^2}-1$ for $b>a>0$ in \eqref{Qi-1995-Xuebao} and simplifying give
\begin{equation*}
\frac\pi2\frac{\ln\bigl(\sqrt{b/a}\,+\sqrt{b/a-1}\,\bigr)}{\sqrt{b(b-a)}\,} \le\int_0^{\pi/2}\frac{\td\theta}{\sqrt{a^2\sin^2\theta+b^2\cos^2\theta}\,} \le\frac\pi2\frac{\arctan\sqrt{b/a-1}\,}{\sqrt{a(b-a)}\,}.
\end{equation*}
Since
\begin{equation}
\int_0^{\pi/2}\frac{\td\theta}{\sqrt{a^2\sin^2\theta+b^2\cos^2\theta}\,} =\int_0^{\pi/2}\frac{\td\theta}{\sqrt{a^2\cos^2\theta+b^2\sin^2\theta}\,},
\end{equation}
the proof of Theorem~\ref{2nd-elliptic-bounds-1} is complete.
\end{proof}

\begin{proof}[Proof of Theorem~\ref{iyengar-elliptic-ineq-thm}]
For $0<t<1$ and $\theta\in\bigl[0,\frac\pi2\bigr]$, let
\begin{equation}
f(\theta)=\sqrt{1-t^2\sin^2\theta}\,.
\end{equation}
Direct calculation yields
\begin{align*}
f'(\theta)&=-\frac{t^2\sin\theta\cos\theta}{\sqrt{1-t^2\sin^2\theta}\,},\\
f''(\theta)&=-\frac{t^2\bigl(t^2\sin^4\theta-\sin^2\theta+\cos^2\theta\bigr)} {\bigl(1-t^2\sin^2\theta\bigr)^{3/2}}\\
&=-\frac{t^2\sin^4\theta\bigl(t^2-1+\cot^4\theta\bigr)} {\bigl(1-t^2\sin^2\theta\bigr)^{3/2}}.
\end{align*}
Hence, the function $f'(\theta)$ has a unique minimum
\begin{equation}
-\frac{t^2\sqrt[4]{1-t^2}\,} {\sqrt{\bigl(1-t^2+\sqrt{1-t^2}\,\bigr)\bigl(1+\sqrt{1-t^2}\,\bigr)}\,}
\end{equation}
at
\begin{equation}
\theta=\arctan\frac1{\sqrt[4]{1-t^2}\,}.
\end{equation}
Therefore, the maximum of $f'(\theta)$ is
\begin{equation}
\lim_{\theta\to0^+}f'(\theta)=\lim_{\theta\to(\pi/2)^-}f'(\theta)=0.
\end{equation}
Moreover, we have
\begin{equation}
f(0)=1\quad\text{and}\quad f\Bigl(\frac\pi2\Bigr)=\sqrt{1-t^2}\,.
\end{equation}
Substituting quantities above into \eqref{gaziy11} and simplifying lead to \eqref{iyengar-elliptic-ineq}. The proof of Theorem~\ref{iyengar-elliptic-ineq-thm} is complete.
\end{proof}

\begin{proof}[Proof of Theorem~\ref{iyengar-elliptic-2nd-ineq-thm}]
For $0<t<1$ and $\theta\in\bigl[0,\frac\pi2\bigr]$, let
\begin{equation}
h(\theta)=\frac1{\sqrt{1-t^2\sin^2\theta}\,}.
\end{equation}
Direct calculation yields
\begin{align*}
h'(\theta)&=\frac{t^2\sin\theta\cos\theta}{\bigl(1-t^2\sin^2\theta\bigr)^{3/2}},\\
h''(\theta)&=-\frac{t^2\bigl(\sin^2\theta-\cos^2\theta-t^2\sin^4\theta-2t^2\cos^2\theta\sin^2\theta\bigr)} {\bigl(1-t^2 \sin ^2\theta\bigr)^{5/2}}\\
&=-\frac{t^2\bigl[t^2\sin^4\theta+2\bigl(1-t^2\bigr)\sin^2\theta-1\bigr]} {\bigl(1-t^2\sin^2\theta\bigr)^{5/2}}.
\end{align*}
Hence, the function $h'(\theta)$ has a unique maximum
\begin{equation}
\frac{\sqrt{\bigl(\sqrt{t^4-t^2+1}\,+t^2-1\bigr)\bigl(1-\sqrt{t^4-t^2+1}\,\bigr)}\,} {\bigl(2-t^2-\sqrt{t^4-t^2+1}\,\bigr)^{3/2}}
\end{equation}
at
\begin{equation}
\theta=\arcsin\frac{\sqrt{\sqrt{t^4-t^2+1}\,+t^2-1}\,}{t}.
\end{equation}
Therefore, the minimum of $h'(\theta)$ is
\begin{equation}
\lim_{\theta\to0^+}h'(\theta)=\lim_{\theta\to(\pi/2)^-}h'(\theta)=0.
\end{equation}
Moreover, we have
\begin{equation}
h(0)=1\quad\text{and}\quad h\Bigl(\frac\pi2\Bigr)=\frac1{\sqrt{1-t^2}\,}.
\end{equation}
Substituting quantities above into \eqref{gaziy11} and simplifying lead to \eqref{iyengar-elliptic-2nd-ineq}. The proof of Theorem~\ref{iyengar-elliptic-2nd-ineq-thm} is complete.
\end{proof}

\section{Remarks}

\begin{rem}
In \cite{amm-556}, it was posed that
\begin{equation}\label{amm-ellip}
\frac\pi6<\int_0^1\frac1{\sqrt{4-x^2-x^3}}\,\td x<\frac{\pi\sqrt2}8.
\end{equation}
In \cite{Garstang}, the inequality \eqref{amm-ellip} was verified by using $4-x^2>4-x^2-x^3>4-2x^2$ on the unit interval $[0,1]$.
\par
In \cite{ellip-gong-ke}, by considering monotonicity and convexity of the function
\begin{gather}
\frac1{\sqrt{4-x^2-x^3}}\,-\frac12+\frac{1-\sqrt2}2\,x^4+\alpha x^3(1-x)
\end{gather}
on $[0,1]$ for undetermined constant $\alpha\ge0$, the inequality
\begin{equation}\label{37ineq}
\frac1{\sqrt{4-x^2-x^3}}\,\ge\frac12+\frac{\sqrt2\,-1}2x^4
+\biggl(\frac{11\sqrt2}8\,-2\biggr)(1-x)x^3
\end{equation}
for $x\in[0,1]$ was established, and then the lower bound in \eqref{amm-ellip}
was improved to
\begin{equation}\label{b1}
\int_0^1\frac1{\sqrt{4-x^2-x^3}}\,\td x>\frac3{10}+\frac{27\sqrt2}{160}\,.
\end{equation}
\par
It was also remarked in \cite{ellip-gong-ke} that if discussing the auxiliary
functions
\begin{gather}
\frac1{\sqrt{4-x^2-x^3}}\,-\frac12+\frac{1-\sqrt2}2\,x^2+\beta(1-x)x^2
\end{gather}
{and}
\begin{gather}
\frac1{\sqrt{4-x^2-x^3}}\,-\frac12+\frac{1-\sqrt2}2\,x^4+\theta(1-x^3)x
\end{gather}
on $[0,1]$, then inequalities
\begin{gather}
\frac1{\sqrt{4-x^2-x^3}}\,\ge\frac12+\frac{\sqrt2\,-1}2x^2
+\biggl(\frac{3\sqrt2}8\,-1\biggr)(1-x)x^2
\end{gather}
{and}
\begin{gather}
\frac1{\sqrt{4-x^2-x^3}}\,\ge\frac12+\frac{\sqrt2\,-1}2x^4
+\biggl(\frac23-\frac{11\sqrt2}{24}\,\biggr)(x^3-1)x
\end{gather}
can be obtained, and then, by integrating on both sides of above two inequalities, the lower bound in \eqref{amm-ellip} may be improved to
\begin{gather}\label{b2}
\int_0^1\frac1{\sqrt{4-x^2-x^3}}\,\td x>\frac14
+\frac{19\sqrt2}{96}
\end{gather}
{and}
\begin{gather}
\int_0^1\frac1{\sqrt{4-x^2-x^3}}\,\td
x>\frac15+\frac{19\sqrt2}{80}\,.\label{b3}
\end{gather}
Numerical computation shows that the lower bound in \eqref{b1} is better than
those in \eqref{b2} and \eqref{b3}.
\par
In \cite{Yu-Kuang-Ye}, by directly proving the inequality \eqref{37ineq} and
\begin{equation}\label{75ineq}
\frac1{\sqrt{4-x^2-x^3}} \le\frac12+\frac{\sqrt2\,-1}2x^2
+\frac{5-4\sqrt2}8\,x^2(1-x) \biggl(\frac{8\sqrt2\,-9}{8\sqrt2\,-10}+x\biggr),
\end{equation}
the inequality \eqref{b1} and an improved upper bound in \eqref{amm-ellip},
\begin{equation}\label{upper}
\int_0^1\frac1{\sqrt{4-x^2-x^3}}\,\td x<\frac{79}{192}+\frac{\sqrt2}{10},
\end{equation}
were obtained.
\par
In \cite{ellip-math-practice}, by considering an auxiliary function
\begin{equation}
\frac1{\sqrt{4-x^2-x^3}}\,-\frac12+\frac{1-\sqrt2}2\,x^2 +\alpha
x^2(1-x)\biggl(\frac{8\sqrt2\,-9}{8\sqrt2\,-10}+x\biggr)
\end{equation}
on $[0,1]$, the sharpness of the inequality \eqref{75ineq} and the following sharp
inequality
\begin{equation}
\frac1{\sqrt{4-x^2-x^3}}\ge\frac12+\frac{\sqrt2\,-1}2x^2 -\frac{1137 \bigl(4
\sqrt{2}\,-5\bigr)}{64\bigl(64-39\sqrt{2}\,\bigr)}(1-x)
\biggl(\frac{8\sqrt2\,-9}{8\sqrt2\,-10}+x\biggr)
\end{equation}
were presented, and then the inequality \eqref{upper} was obtained by integrating
on both sides of \eqref{75ineq}.
\end{rem}

\begin{rem}
Integrating on both sides of \eqref{81-ineq} with respect to $\theta$ from $0$ to $\frac\pi2$ yields
\begin{equation}\label{ellip-ellipse}
\frac{\pi}6(2a+b)<\int_0^{\pi/2}\sqrt{a^2\sin^2\theta+b^2\cos^2\theta}\,\td
\theta\le\frac{\pi}6(a+2b).
\end{equation}
When $b\ge 7a$, the right-hand side of the inequality \eqref{ellip-ellipse} is stronger than
\begin{equation}
\frac\pi4(a+b)\le\int_0^{\pi/2}\sqrt{a^2\sin^2\theta+b^2\cos^2\theta}\,\td \theta \le\frac\pi4\sqrt{2(a^2+b^2)}
\end{equation}
which can be obtained by using some properties of definite integral.
\end{rem}

\begin{rem}
The complete elliptic integral of the third kind may be defined for $0<t<1$ as
\begin{align}
I\!I(t,h)&=\int_0^{\pi/2}\frac{\td\theta}{(1+h\sin^2\theta) \sqrt{1-t^2\sin^2\theta}}. \label{4}
\end{align}
\par
By the way, some other estimates for complete elliptic integrals, obtained by using Tchebycheff's integral inequality in \cite{ellip-huang}, are mentioned below:
\begin{gather}
\frac{\pi\arcsin t}{ 2t}<F(t) <\frac{\pi}{4t}\ln\frac{1+t}{1-t}; \label{9}\\
E(t)<\frac{16-4t^2-3t^4}{ 4(4+t^2)}F(t); \label{10}\\
F(t)<\Bigl(1+\frac{h}{ 2}\Bigr)I\!I(t,h), \quad -1<h<0\quad \textup{or}\quad h>\frac{t^2}{ 2-3t^2} >0; \label{11}\\
I\!I(t,h) E(t)>\frac{\pi^2}{ 4\sqrt{1+h}}, \quad -2<2h<t^2; \label{12}\\
E(t)\ge \frac{16-28t^2+9t^4}{ 4(4-5t^2)}F(t), \quad t^2\le \frac{2}{3}. \label{13}
\end{gather}
For $0<2h<t^2$, the inequality \eqref{11} is reversed. For $h>\frac{t^2}
{2-3t^2}>0$, the inequality \eqref{12} is reversed.
\par
As concrete examples, the following estimates of the complete elliptic integrals are also deduced in \cite{ellip-huang}:
\begin{gather}
\frac{\pi^2}{ 4\sqrt 2}<\int_0^{\pi/2}\biggl(1-\frac{\sin^2x}{ 2}
\biggr)^{-1/2}\td x<\frac{\pi\ln(1+\sqrt 2\,)}{\sqrt 2},
\label{14}\\
\int_0^{\pi/2}\biggl(1+\frac{\cos x}{ 2}\biggr)^{-1}\td x
<\frac{\pi(\ln 3-\ln2)}{ 2},
\label{15}\\
\int_0^{\pi/2}\biggl(1-\frac{\sin x}{ 2}\biggr)^{-1}\td x
=\int_{\pi/2}^\pi\biggl(1+\frac{\cos x}{ 2}\biggr)^{-1}\td x
>\frac{\pi\ln 2}{ 2}.
\label{16}
\end{gather}
These results are better than those in \cite[p.~607]{kuang-2nd}.
\end{rem}

\begin{rem}
It is noted that some new results on complete elliptic integrals are obtained
in \cite{turan-baricz} recently. It was pointed in \cite{turan-baricz} that the right-hand side inequality in \eqref{9} is a recovery of \cite[Theorem~3.10]{avv}. In \cite{turan-baricz}, the inequality \eqref{9} was also generalized to the case of generalized complete elliptic integrals by the same method as in \cite{construct, ellip-huang}.
\par
In \cite{turan-baricz-proc}, some of the results in \cite{turan-baricz} were further improved.
\end{rem}

\end{document}